\documentclass[letterpaper]{article}
\usepackage{amsmath}
\usepackage{amsfonts}

\newtheorem{theorem}{Theorem}

\newtheorem{corollary}[theorem]{Corollary}

\newtheorem{problem}{Problem}
\newtheorem{proposition}[theorem]{Proposition}

\newenvironment{proof}[1][Proof]{\noindent\textbf{#1.} }{\ \rule{0.5em}{0.5em}}
\input{tcilatex}

\begin{document}

\title{Universal Taylor series, conformal mappings and boundary behaviour}
\date{}
\author{Stephen J. Gardiner}
\maketitle

\begin{abstract}
A holomorphic function $f$ on a simply connected domain $\Omega $ is said to
possess a universal Taylor series about a point in $\Omega $ if the partial
sums of that series approximate arbitrary polynomials on arbitrary compacta $%
K$ outside $\Omega $ (provided only that $K$ has connected complement). This
paper shows that this property is not conformally invariant, and, in the
case where $\Omega $ is the unit disc, that such functions have extreme
angular boundary behaviour.
\end{abstract}

\section{Introduction\protect\footnotetext{%
\noindent 2010 \textit{Mathematics Subject Classification }30K05, 30B30,
30E10, 31A05. \newline
This research was supported by Science Foundation Ireland under Grant
09/RFP/MTH2149.}}

Let $f$ be a holomorphic function on a simply connected proper subdomain $%
\Omega $ of the complex plane $\mathbb{C}$, let $\xi \in \Omega $ and $%
S_{N}(f,\xi )(z)$ denote the partial sum $\sum_{n=0}^{N}a_{n}(z-\xi )^{n}$
of the Taylor series of $f$\ about $\xi $. We call this series \textit{%
universal} and write $f\in \mathcal{U}(\Omega ,\xi )$\textit{\ }if, for
every compact set $K\subset \mathbb{C}\backslash \Omega $ that has connected
complement and every continuous function $g:K\rightarrow \mathbb{C}$ that is
holomorphic on $K^{\circ }$, there is a subsequence $(S_{N_{k}}(f,\xi ))$
that converges uniformly to $g$ on $K$. It is known that possession of such
universal expansions is a generic property of holomorphic functions on
simply connected domains (that is, $\mathcal{U}(\Omega ,\xi )$ is a dense $%
G_{\delta }$ subset of the space of all holomorphic functions on $\Omega $
endowed with the topology of local uniform convergence \cite{Ne96}, \cite%
{Ne97}) and that the collection $\mathcal{U}(\Omega ,\xi )$ is independent
of the choice of the centre of expansion $\xi $ (see \cite{MN}, \cite{MVY}).

However, significant questions remain open. A fundamental issue concerns
conformal invariance:

\begin{problem}
\label{P1}If $F:\Omega _{0}\rightarrow \Omega $ is a conformal mapping,
where $\Omega _{0}$ and $\Omega $ are simply connected domains, and if $f\in 
\mathcal{U}(\Omega ,\xi )$, does it follow that $f\circ F\in \mathcal{U}%
(\Omega _{0},F^{-1}(\xi ))$?
\end{problem}

Another question concerns boundary behaviour, about which there is a growing
literature \cite{Ne97}, \cite{CM}, \cite{Me00}, \cite{MN}, \cite{Cos05}, 
\cite{MVY}, \cite{AC}, \cite{BBCP}, \cite{G12}. For example, in the case of
the unit disc $\mathbb{D}$, it is known that if $f\in \mathcal{U}(\mathbb{D}%
,0)$, then $f$ does not belong to the Nevanlinna class (see \cite{MNP}) and
there is a residual subset $Z$ of the unit circle $\mathbb{T}$ such that the
set $\{f(r\zeta ):0<r<1\}$\textit{\ }is unbounded for every $\zeta \in Z$
(see \cite{Bay}). However, little progress has yet been made on the natural
question:

\begin{problem}
\label{P2}What can be said about the angular boundary behaviour of functions
in $\mathcal{U}(\mathbb{D},0)$?
\end{problem}

I am grateful to Vassili Nestoridis for alerting me to the fact that Problem %
\ref{P1} had remained unresolved, and to George Costakis for drawing my
attention to Problem \ref{P2}. The answers are given below. Let $S$ denote
the strip $\{z\in \mathbb{C}:-1<\func{Re}z<1\}$.

\begin{theorem}
\label{con}There is a function $f\in \mathcal{U}(S,0)$ with the following
properties:\newline
(i) for any conformal mapping $F:\mathbb{D}\rightarrow S$ we have $f\circ
F\notin \mathcal{U}(\mathbb{D},F^{-1}(0))$;\newline
(ii) there exist conformal mappings $F:S\rightarrow S$ such that $f\circ
F\notin \mathcal{U}(S,F^{-1}(0))$.
\end{theorem}

We define angular approach regions at a point $\zeta \in \mathbb{T}$ by%
\begin{equation*}
\Gamma _{\alpha }^{t}(\zeta )=\{z:\left\vert z-\zeta \right\vert <\alpha
(1-\left\vert z\right\vert )<\alpha t\}\text{ \ \ \ }(\alpha >1,0<t\leq 1).
\end{equation*}%
A boundary point $\zeta $ is called a \textit{Fatou point }of a holomorphic
function $f$ on $\mathbb{D}$\ if $\lim_{z\rightarrow \zeta ,z\in \Gamma
_{\alpha }^{1}(\zeta )}f(z)$ exists finitely for all $\alpha $. At the
opposite extreme, $\zeta $ is called a \textit{Plessner point }of $f$ if $%
f(\Gamma _{\alpha }^{t}(\zeta ))$ is dense in $\mathbb{C}$ for all $\alpha $
and $t$. Plessner's theorem says that, for any holomorphic function $f$ on $%
\mathbb{D}$, almost every point of $\mathbb{T}$ is either a Fatou point or a
Plessner point of $f$ (see Theorem 6.13 in \cite{Pomm}). Our next result
shows that universal Taylor series have extreme angular boundary behaviour.

\begin{theorem}
\label{one}If $f\in \mathcal{U}(\mathbb{D},0)$, then almost every point of $%
\mathbb{T}$ is a Plessner point of $f$.
\end{theorem}

An easy consequence of Theorem \ref{one} is the following Baire category
analogue, which strengthens the result of Bayart mentioned earlier.

\begin{corollary}
\label{radial}If $f\in \mathcal{U}(\mathbb{D},0)$, then there is a residual
subset $Z$ of $\mathbb{T}$\textit{\ }such that $\{f(r\zeta ):0<r<1\}$\ is
dense in\textit{\ }$\mathbb{C}$ for every $\zeta \in Z$.
\end{corollary}

It turns out that the above solutions to Problems \ref{P1} and \ref{P2} both
emerge from the same non-trivial potential theoretic insight, which we will
now describe. The Poisson kernel for $\mathbb{D}$ is given by 
\begin{equation*}
P(z,\zeta )=\frac{1-\left\vert z\right\vert ^{2}}{\left\vert z-\zeta
\right\vert ^{2}}\text{ \ \ \ }(z\in \mathbb{D},\zeta \in \mathbb{T}).
\end{equation*}%
A set $E\subset \mathbb{D}$ is said to be \textit{minimally thin at a point }%
$\zeta \in \mathbb{T}$ if there is a (Green) potential $u$ on $\mathbb{D}$
such that $u\geq P(\cdot ,\zeta )$ on $E$. For example, if $D\subset \mathbb{%
D}$ is a disc that is internally tangent to $\mathbb{T}$ at a point $\zeta $%
, then $\mathbb{D}\backslash D$ is minimally thin at $\zeta $. This follows
from the facts that $D$ is of the form $\{P(\cdot ,\zeta )>c\}\ $for some $%
c>0$, and that $\min \{P(\cdot ,\zeta ),c\}$ is a potential on $\mathbb{D}$
since its greatest harmonic minorant is readily seen to be $0$.\ More
generally (see Theorem 2 in \cite{BD} and Theorem 9.5.5(iii) in \cite{AG}),
if $\psi :[0,1]\rightarrow \lbrack 0,1]$ is increasing, then the set $\{z\in 
\mathbb{D}:\func{Re}z>1-\psi (\left\vert \func{Im}z\right\vert )\}$ is
minimally thin at $1$ if and only if $\int_{0}^{1}t^{-2}\psi (t)dt<\infty $.
An introduction to the notion of minimal thinness may be found in Chapter 9
of the book \cite{AG}.

The key underlying result in this paper is as follows. We abbreviate $%
S_{N}(f,0)$ to $S_{N}$.

\begin{theorem}
\label{two}Let $f$ be a holomorphic function on $\mathbb{D}$ and $h$ be a
positive harmonic function on $\mathbb{D}$ such that the set $\{\left\vert
f\right\vert \geq e^{h}\}$ is minimally thin at $\zeta _{0}\in \mathbb{T}$.
If $(S_{N_{k}})$ is uniformly bounded on an open arc of $\mathbb{T}$ that
contains $\zeta _{0}$, then $(e^{-h}S_{N_{k}})$ is uniformly bounded on a
set of the form $\mathbb{D}\backslash E$, where $E\mathbb{\ }$is minimally
thin at $\zeta _{0}$. In the particular case where $h$ is constant, we can
thus conclude that $(S_{N_{k}})$ is uniformly bounded on $\mathbb{D}%
\backslash E$.
\end{theorem}

\begin{corollary}
\label{cor}If $f\in \mathcal{U}(\mathbb{D},0)$ and $h$ is a positive
harmonic function on $\mathbb{D}$, then there is at most one point of $%
\mathbb{T}$ at which the set $\{\left\vert f\right\vert \geq e^{h}\}$ is
minimally thin.
\end{corollary}

Bayart \cite{Bay} has shown that, if $f\in \mathcal{U}(\mathbb{D},0)$ and $%
a>0$, then there is at most one point $\zeta $ of $\mathbb{T}$ such that $%
\left\vert f\right\vert <a$ on a disc internally tangent to $\mathbb{T}$ at $%
\zeta $. Corollary \ref{cor} is a significantly stronger result, and this
extra strength is crucial for our purposes. The \textquotedblleft one
point\textquotedblright\ in Corollary \ref{cor} can actually arise. This
follows by choosing the set $A$ in the following result so that $\mathbb{D}%
\backslash A$ is minimally thin at $1$.

\begin{proposition}
\label{ex}Let $A\subset \mathbb{D}$, where $\overline{A}\cap \mathbb{T}%
=\{1\} $, and let $w:\mathbb{D}\rightarrow (1,\infty )$ be a continuous
function such that $w(z)\rightarrow \infty $ as $z\rightarrow 1$. Then there
exists $f\in \mathcal{U}(\mathbb{D},0)$ such that $\left\vert f\right\vert
\leq w$ on $A$. In particular, this is true for $w=e^{h}$, where $h$ is a
positive harmonic function on $\mathbb{D}$ that tends to $\infty $ at $1$.
\end{proposition}

Let $D$ be a disc contained in $\mathbb{D}$ that is internally tangent to $%
\mathbb{T}$ at the point $1$. As noted in the proof of Proposition 5.6 in 
\cite{MN}, no member of $\mathcal{U}(\mathbb{D},0)$, when restricted to $D$,
can have a limit at $1$. However, by the above proposition, there exists $f$
in $\mathcal{U}(\mathbb{D},0)$ satisfying $\left\vert f(z)\right\vert \leq
\left\vert z-1\right\vert ^{-1/2}$ on $D$, whence $(z-1)f(z)\rightarrow 0$
as $z\rightarrow 1$ in $D$. Thus the function $z\mapsto (z-1)f(z)$ does not
belong $\mathcal{U}(\mathbb{D},0)$. This answers a question of Costakis \cite%
{Cos08}, who had asked whether the property of having a universal Taylor
series is preserved under multiplication by non-constant polynomials.
Similarly, no antiderivative of this function $f$ can belong to $\mathcal{U}(%
\mathbb{D},0)$. This gives a negative answer to another question of Costakis
(private communication), about whether antiderivatives of universal Taylor
series are necessarily universal. (The corresponding question for
derivatives remains open.) Costakis has also observed that Theorem \ref{one}
above and Theorem 1.2 of \cite{BR} together show that each member of $%
\mathcal{U}(\mathbb{D},0)$ must tend to $\infty $ along some path to the
boundary.

We will prove Theorem \ref{two} in the next section and subsequently proceed
to the remaining proofs.

\section{Proof of Theorem \protect\ref{two}}

Let $D(z,r)$ denote the open disc of centre $z$ and radius $r$, let $C(K)$
denote the space of real-valued continuous functions on a compact set $K$,
and let $\widehat{\mathbb{C}}=\mathbb{C}\cup \{\infty \}$ denote the
extended complex plane. If $U\subset \widehat{\mathbb{C}}$ is open, we
denote by $G_{U}(\cdot ,\zeta )$ the Green function for $U$ with pole at $%
\zeta \in U$, and assign this function the value $0$ outside $U$.

Now let $f$ be a holomorphic function on $\mathbb{D}$ and $h$ be a positive
harmonic function on $\mathbb{D}$ such that the set $\{\left\vert
f\right\vert \geq e^{h}\}$ is minimally thin at $\zeta _{0}\in \mathbb{T}$.
We define 
\begin{equation*}
U=\left\{ z\in D(3\zeta _{0}/4,1/4):\left\vert f(z)\right\vert
<e^{h}\right\} .
\end{equation*}%
Then $U$ is open and $\overline{U}\cap \mathbb{T=\{\zeta }_{0}\}$. Also, $%
\mathbb{D}\backslash U$ is minimally thin at $\zeta _{0}$, since 
\begin{equation*}
\mathbb{D}\backslash U=\left[ \mathbb{D}\backslash D(3\zeta _{0}/4,1/4)%
\right] \cup \{\left\vert f\right\vert \geq e^{h}\}
\end{equation*}%
and the union of two sets that are minimally thin at $\zeta _{0}$ is also
minimally thin at $\zeta _{0}$. Let $\mu _{z}$ denote harmonic measure for $%
U $ and $z\in U$. For each $z\in U$ we define a modified measure $\mu
_{z}^{\ast }$ on $\partial U$ by writing 
\begin{equation*}
d\mu _{z}^{\ast }(\zeta )=\frac{\log (1/\left\vert \zeta \right\vert )}{\log
(1/\left\vert z\right\vert )}d\mu _{z}(\zeta )\text{ \ \ \ }(\zeta \in
\partial U).
\end{equation*}%
These are probability measures since the function $\zeta \mapsto \log
(1/\left\vert \zeta \right\vert )$ is harmonic on $\mathbb{C}\backslash
\{0\} $.

We will make use of some key facts about minimal thinness from \cite{AG}.
The first of these, Theorem 9.6.2, describes how the minimal thinness of $%
\mathbb{D}\backslash U$ at $\zeta _{0}$ affects the behaviour of positive
superharmonic functions $v$ on $U$ near $\zeta _{0}$. Specifically, it tells
us that, for each such $v$, there is a set $E(v)\subset \mathbb{D}$,
minimally thin at $\zeta _{0}$, and a number $l(v)\in (0,\infty ]$, such
that 
\begin{equation*}
\frac{v(z)}{\log (1/\left\vert z\right\vert )}=\frac{v(z)}{G_{\mathbb{D}%
}(z,0)}\rightarrow l(v)\text{ \ \ \ }(z\rightarrow \zeta _{0},z\in \mathbb{D}%
\backslash E(v)).
\end{equation*}

If $\phi \in C(\partial U)$, then $\left\vert \int \phi d\mu _{z}^{\ast
}\right\vert \leq \max_{\partial U}\left\vert \phi \right\vert $ for all $%
z\in U$. By considering separately the cases where $v$ is given by 
\begin{equation*}
z\mapsto \int \phi ^{\pm }(\zeta )\log (1/\left\vert \zeta \right\vert )d\mu
_{z}(\zeta ),
\end{equation*}%
we now see that there is a set $E_{\phi }\subset \mathbb{D}$, minimally thin
at $\zeta _{0}$, and a number $l_{\phi }\in \mathbb{R}$, such that 
\begin{eqnarray*}
\int \phi d\mu _{z}^{\ast } &=&\frac{1}{\log (1/\left\vert z\right\vert )}%
\int \phi (\zeta )\log (1/\left\vert \zeta \right\vert )d\mu _{z}(\zeta ) \\
&\rightarrow &l_{\phi }\text{ \ \ \ }(z\rightarrow \zeta _{0},z\in \mathbb{D}%
\backslash E_{\phi }).
\end{eqnarray*}%
Now let $(\phi _{n})$ be a dense sequence in $C(\partial U)$. Lemma 9.3.1 in 
\cite{AG} allows us to construct a set $E^{\ast }\subset \mathbb{D}$,
minimally thin at $\zeta _{0}$, and a sequence of positive numbers $(\rho
_{n})$, decreasing to $0$, such that%
\begin{equation*}
E_{\phi _{n}}\cap D(\zeta _{0},\rho _{n})\subset E^{\ast }\text{ \ \ \ }%
(n\in \mathbb{N}).
\end{equation*}%
Thus, for each $n\in \mathbb{N}$, the function $z\mapsto \left( \int \phi
_{n}d\mu _{z}^{\ast }\right) $ converges to a finite limit as\ $z\rightarrow
\zeta _{0}$ in $\mathbb{D}\backslash E^{\ast }$. It follows that the limit
measure 
\begin{equation}
\nu _{0}=\lim_{z\rightarrow \zeta _{0},z\in \mathbb{D}\backslash E^{\ast
}}\mu _{z}^{\ast }  \label{no}
\end{equation}%
exists in the sense of $w^{\ast }$-convergence of measures. (The argument we
have used in this paragraph can be regarded as a minimal fine topology\
analogue of that used in Doob \cite{Do} to construct fine harmonic measure
at an irregular boundary point of a domain.)

Clearly $\nu _{0}$ is a probability measure on $\partial U$. We will now
show that $\nu _{0}(\{\zeta _{0}\})=0$. Since $\mathbb{D}\backslash U$ is
minimally thin at $\zeta _{0}$, we can combine Theorems 9.2.7, 9.3.3(ii) and
equation (9.2.4) in \cite{AG} to see that there is a Green potential $v_{0}$
on $\mathbb{D}$ and a set $E_{0}\subset \mathbb{D}$, minimally thin at $%
\zeta _{0}$, such that 
\begin{equation*}
\frac{v_{0}(z)}{\log (1/\left\vert z\right\vert )}\rightarrow \infty \text{
\ \ \ }(z\rightarrow \zeta _{0},z\in \mathbb{D}\backslash U)
\end{equation*}%
and%
\begin{equation*}
\frac{v_{0}(z)}{\log (1/\left\vert z\right\vert )}\rightarrow 1\text{ \ \ \ }%
(z\rightarrow \zeta _{0},z\in U\backslash E_{0}).
\end{equation*}%
Let $\varepsilon >0$. Then there exists $r>0$ such that 
\begin{equation*}
v_{0}(z)>\varepsilon ^{-1}\log \left( 1/\left\vert z\right\vert \right) 
\text{ \ \ \ \ }(z\in (\mathbb{D}\backslash U)\cap D(\zeta _{0},r))
\end{equation*}%
and%
\begin{equation*}
v_{0}(z)<2\log \left( 1/\left\vert z\right\vert \right) \text{ \ \ \ \ }%
(z\in (U\backslash E_{0})\cap D(\zeta _{0},r)).
\end{equation*}%
Hence 
\begin{eqnarray*}
\mu _{z}^{\ast }(D(\zeta _{0},r)\cap \partial U) &=&\frac{1}{\log
(1/\left\vert z\right\vert )}\int_{D(\zeta _{0},r)\cap \partial U}\log
(1/\left\vert \zeta \right\vert )d\mu _{z}(\zeta ) \\
&\leq &\frac{1}{\log (1/\left\vert z\right\vert )}\int_{D(\zeta _{0},r)\cap
\partial U}\varepsilon v_{0}d\mu _{z} \\
&\leq &\varepsilon \frac{v_{0}(z)}{\log (1/\left\vert z\right\vert )} \\
&<&2\varepsilon \text{ \ \ \ }(z\in (U\backslash E_{0})\cap D(\zeta _{0},r)),
\end{eqnarray*}%
by the superharmonicity of $v_{0}$, and so $\nu _{0}(D(\zeta _{0},r)\cap
\partial U)\leq 2\varepsilon $. Since $\varepsilon >0$ was arbitrary, $\nu
_{0}(\{\zeta _{0}\})=0$ as claimed.

Let $I$ be an open arc of $\mathbb{T}$ containing $\zeta _{0}$ on which $%
(S_{N_{k}})$ is uniformly bounded, and let $\psi _{j}:\partial U\backslash
\{\zeta _{0}\}\rightarrow \mathbb{R}$ be the function given by 
\begin{equation}
\psi _{j}(z)=\left\{ 
\begin{array}{cc}
-\frac{1}{2} & (\left\vert z\right\vert <1-\frac{1}{j}) \\ 
G_{\widehat{\mathbb{C}}\backslash \overline{I}}(z,\infty )/\log
(1/\left\vert z\right\vert ) & (1>\left\vert z\right\vert \geq 1-\frac{1}{j})%
\end{array}%
\right. .  \label{psi}
\end{equation}%
It is easy to check that $G_{\widehat{\mathbb{C}}\backslash \overline{I}%
}(z,\infty )/\log (1/\left\vert z\right\vert )$ has a finite (positive)
limit as $z\rightarrow \zeta _{0}$ in $\mathbb{D}$. Thus $\psi _{j}$,
extended by this limiting value, is upper semicontinuous and bounded above
on $\partial U$. Further, $\psi _{j}\downarrow -1/2$ on $\partial
U\backslash \{\zeta _{0}\}$ as $j\rightarrow \infty $. Hence we can find $%
j_{0}\in \mathbb{N}$ such that 
\begin{equation}
\int \psi _{j_{0}}d\nu _{0}<0.  \label{neg}
\end{equation}

For each $k\in \mathbb{N}$ we define the subharmonic function 
\begin{equation}
u_{k}=\frac{1}{N_{k}}\log \left\vert S_{N_{k}}-f\right\vert \text{ \ \ \ on }%
\mathbb{D}.  \label{uk0}
\end{equation}%
Since $S_{N_{k}}-f$ has a zero of order (at least) $N_{k}$ at $0$, the
function $u_{k}(z)-\log \left\vert z\right\vert $\ is also subharmonic on $%
\mathbb{D}$. Further, $\lim \sup_{k\rightarrow \infty }u_{k}\leq 0$. Thus it
follows from the maximum principle that 
\begin{equation*}
\underset{k\rightarrow \infty }{\lim \sup }~u_{k}(z)\leq \log \left\vert
z\right\vert \text{ \ \ on }\mathbb{D}.
\end{equation*}%
Hence (see Corollary 5.7.2 in \cite{AG}) we can choose $k_{0}\in \mathbb{N}$
such that 
\begin{equation}
u_{k}(z)\leq \frac{\log \left\vert z\right\vert }{2}\text{ \ \ \ }%
(\left\vert z\right\vert \leq 1-\frac{1}{j_{0}},k\geq k_{0}).  \label{lu}
\end{equation}%
Also, by Bernstein's lemma (see Theorem 5.5.7 in \cite{Ran}),%
\begin{equation*}
\log \left\vert S_{N_{k}}\right\vert \leq N_{k}G_{\widehat{\mathbb{C}}%
\backslash \overline{I}}(\cdot ,\infty )+\log \left( \sup\nolimits_{%
\overline{I}}\left\vert S_{N_{k}}\right\vert \right) .
\end{equation*}%
We know that there exists $a\geq 1$ such that $\left\vert
S_{N_{k}}\right\vert \leq a$ on $I$ for all $k$. On $\overline{U}\cap 
\mathbb{D}$ we thus have%
\begin{eqnarray}
u_{k} &\leq &\frac{1}{N_{k}}\log \left( 2\max \left\{ \left\vert
S_{N_{k}}\right\vert ,\left\vert f\right\vert \right\} \right)  \notag \\
&\leq &\frac{1}{N_{k}}\left( \log 2+\max \left\{ N_{k}G_{\widehat{\mathbb{C}}%
\backslash \overline{I}}(\cdot ,\infty )+\log a,h\right\} \right)  \notag \\
&\leq &G_{\widehat{\mathbb{C}}\backslash \overline{I}}(\cdot ,\infty )+\frac{%
\log 2a+h}{N_{k}}.  \label{uk}
\end{eqnarray}%
Using the subharmonicity of $u_{k}-(\log 2a+h)/N_{k}$ and its upper
boundedness on $U$, and then (\ref{psi}), (\ref{lu}) and (\ref{uk}), we see
that%
\begin{eqnarray}
u_{k}(z)-\frac{\log 2a+h(z)}{N_{k}} &\leq &\int_{\mathbb{D}\cap \partial
U}\left( u_{k}-\frac{\log 2a+h}{N_{k}}\right) d\mu _{z}  \notag \\
&\leq &\int_{\partial U}\psi _{j_{0}}(\zeta )\log (1/\left\vert \zeta
\right\vert )d\mu _{z}(\zeta )  \notag \\
&=&\log (1/\left\vert z\right\vert )\int_{\partial U}\psi _{j_{0}}d\mu
_{z}^{\ast }\text{ \ \ \ }(z\in U,k\geq k_{0}).  \label{est}
\end{eqnarray}%
By (\ref{no}), the upper semicontinuity of $\psi _{j_{0}}$, and (\ref{neg}),
there exists $r_{1}\in (0,1)$ such that%
\begin{equation}
\int_{\partial U}\psi _{j_{0}}d\mu _{z}^{\ast }<0\text{ \ \ \ }(z\in U\cap
D(\zeta _{0},r_{1})\backslash E^{\ast }).  \label{neg2}
\end{equation}%
Combining (\ref{est}) and (\ref{neg2}) with (\ref{uk0}), we see that%
\begin{equation*}
e^{-h}\left\vert S_{N_{k}}-f\right\vert \leq 2a\text{ \ \ on \ }U\cap
D(\zeta _{0},r_{1})\backslash E^{\ast }\text{ \ when \ }k\geq k_{0},
\end{equation*}%
and the conclusion of Theorem \ref{two} follows on defining 
\begin{equation*}
E=(\mathbb{D}\backslash U)\cup (\mathbb{D}\backslash D(\zeta
_{0},r_{1}))\cup E^{\ast },
\end{equation*}%
which is minimally thin at $\zeta _{0}$.

\section{The remaining proofs}

\begin{proof}[Proof of Corollary \protect\ref{cor}]
Let $f\in \mathcal{U}(\mathbb{D},0)$ and suppose that, for some positive
harmonic function $h$ on $\mathbb{D}$, the set $\{\left\vert f\right\vert
\geq e^{h}\}$ is minimally thin at two distinct points $\zeta _{1},\zeta
_{2}\in \mathbb{T}$. Further, let $I$ be an open arc of $\mathbb{T}$
containing $\zeta _{1}$ and $\zeta _{2}$ such that $\overline{I}\neq \mathbb{%
T}$. In view of the Poisson integral representation of positive harmonic
functions on $\mathbb{D}$ we can easily modify $h$ to obtain another such
function $h_{1}$ that vanishes continuously on a closed subarc $I_{1}$ of $I$
lying between $\zeta _{1}$ and $\zeta _{2}$ and such that the set $%
\{\left\vert f\right\vert \geq e^{h_{1}}\}$ remains minimally thin at $\zeta
_{1},\zeta _{2}$. By universality we can find a subsequence $(S_{N_{k}})$
that is uniformly convergent to $0$ on the set $\{r\zeta :\zeta \in 
\overline{I},1\leq r\leq 2\}$. Theorem \ref{two} then tells us that there is
a set $E\subset \mathbb{D}$, which is minimally thin at both $\zeta _{1}$
and $\zeta _{2}$, such that $(e^{-h_{1}}S_{N_{k}})$ is uniformly bounded on $%
\mathbb{D}\backslash E$. By Theorem 8 of \cite{LF} we can choose line
segments $L_{1},L_{2}\subset \mathbb{D}\backslash E$ with endpoints at $%
\zeta _{1},\zeta _{2}$, respectively. Since $(S_{N_{k}})$ is locally
uniformly convergent on $\mathbb{D}$, it follows from the maximum principle
that $(\log \left\vert S_{N_{k}}\right\vert -h_{1})$ is uniformly bounded on
a domain $\Omega $ whose boundary is contained in the union of $%
L_{1},L_{2},I $ and a suitable closed line segment in $\mathbb{D}$\ joining $%
L_{1}$ to $L_{2}$. Hence $(S_{N_{k}})$ is uniformly bounded on the set $%
\omega =\{r\zeta :\zeta \in I_{1},0<r<2\}$. This leads to the conclusion
that $S_{N_{k}}\rightarrow 0$ on $\omega $ and thus $f\equiv 0$, which is a
contradiction.
\end{proof}

\bigskip

\begin{proof}[Proof of Theorem \protect\ref{one}]
Let $f\in \mathcal{U}(\mathbb{D},0)$ and suppose that the set of Plessner
points of $f$ does not have full arclength measure. We fix $\alpha >1$ and $%
0<t\leq 1$, and define%
\begin{equation*}
J_{a}=\{\zeta \in \mathbb{T}:\left\vert f\right\vert \leq a\text{ \ on \ }%
\Gamma _{\alpha }^{t}(\zeta )\}\text{ \ \ \ \ }(a>0).
\end{equation*}%
By Plessner's theorem the set $J_{a}$ will then have positive arclength
measure provided we choose $a$ large enough. Let $F=\cup _{\zeta \in
J_{a}}\Gamma _{\alpha }^{t}(\zeta )$. The set $\mathbb{D}\backslash F$ is
then minimally thin at almost every point of $J_{a}$, by Lemma 9.7.5 of \cite%
{AG} and the conformal invariance of minimal thinness. This leads to a
contradiction, in view of the Corollary \ref{cor} and the fact that $%
\left\vert f\right\vert \leq a$ on $F$.
\end{proof}

\bigskip

\begin{proof}[Proof of Corollary \protect\ref{radial}]
Let $f\in \mathcal{U}(\mathbb{D},0)$, and let $Z$ denote the set of all $%
\zeta \in \mathbb{T}$\textit{\ }such that $\{f(r\zeta ):0<r<1\}$\ is dense in%
\textit{\ }$\mathbb{C}$. If $\zeta \in \mathbb{T}\backslash Z$, then we can
choose $p\in \mathbb{Q}+i\mathbb{Q}$ and a positive rational number $q$ such
that 
\begin{equation}
\{f(r\zeta ):0<r<1\}\subset \mathbb{C}\backslash D(p,q).  \label{inc}
\end{equation}%
We write $E_{p,q}$ for the collection of all points $\zeta \in \mathbb{T}$
satisfying (\ref{inc}). Thus $E_{p,q}$ is closed and $\mathbb{T}\backslash
Z=\cup _{p,q}E_{p,q}$. If $Z$ were not residual, then there would exist $p,q$
as above such that $E_{p,q}$ has non-empty interior $J$ relative to $\mathbb{%
T}$. It follows that $f$ does not take values in $D(p,q)$ on the sector $%
\{r\zeta :0<r<1,\zeta \in J\}$. This contradicts the conclusion of Theorem %
\ref{one}, so $Z$ must be residual.
\end{proof}

\bigskip

\begin{proof}[Proof of Proposition \protect\ref{ex}]
Without loss of generality we may assume that $A$ is closed relative to $%
\mathbb{D}$ and that $\overline{A}\cup \overline{D}(0,n/(n+1))$ has
connected complement for each $n\in \mathbb{N}$. By Lemma 2.1 of \cite{Ne96}
there is a countable collection $\mathcal{K}$, of compact sets $K\subset 
\mathbb{C}\backslash \mathbb{D}$ with connected complement, having the
following property: if $L\subset \mathbb{C}\backslash \mathbb{D}$ is compact
and $\mathbb{C}\backslash L$ is connected, then $L\subseteq K$ for some $%
K\in \mathcal{K}$. It is easy to see that $\overline{A}\cup \overline{D}%
(0,n/(n+1))\cup K$ has connected complement for every $K\in \mathcal{K}$.\
Now let $\mathcal{P}$ be the collection of all complex polynomials with
coefficients in $\mathbb{Q}+i\mathbb{Q}$, let $((K_{n},p_{n}))$ be an
enumeration of $\mathcal{K}\times \mathcal{P}$, and let $d_{n}=\max_{z\in
K_{n}}\left\vert z\right\vert $. We inductively define a sequence of
polynomials $(q_{n})$ as follows.

Since $w(z)\rightarrow \infty $ as $z\rightarrow 1$ we can choose $n_{1}\in 
\mathbb{N}$ large enough so that $\left\vert z^{n_{1}}p_{1}(1)\right\vert
\leq w(z)/2^{2}$ on $A\cup \overline{D}(0,1/2)$. We then define 
\begin{equation*}
p_{1}^{\ast }(z)=\left\{ 
\begin{array}{cc}
z^{-n_{1}}p_{1}(z) & \text{if }\left\vert z\right\vert \geq 1 \\ 
p_{1}(1) & \text{if }\left\vert z\right\vert <1%
\end{array}%
\right. ,
\end{equation*}%
use Mergelyan's theorem to choose a polynomial $q_{1}^{\ast }$ such that 
\begin{equation*}
\left\vert q_{1}^{\ast }-p_{1}^{\ast }\right\vert <(2^{2}d_{1}^{n_{1}})^{-1}%
\text{ \ \ \ on \ }\overline{A}\cup \overline{D}(0,1/2)\cup K_{1},
\end{equation*}%
and define $q_{1}(z)=z^{n_{1}}q_{1}^{\ast }(z)$. Since $w\geq 1$ and $%
d_{1}\geq 1$, we have 
\begin{eqnarray*}
\left\vert q_{1}(z)\right\vert &\leq &\left\vert z^{n_{1}}\right\vert
\left\vert q_{1}^{\ast }(z)-p_{1}^{\ast }(z)\right\vert +\left\vert
z^{n_{1}}p_{1}(1)\right\vert \\
&\leq &2^{-2}+2^{-2}w(z) \\
&\leq &2^{-1}w(z)\text{ \ \ \ \ \ \ \ \ \ }(z\in A\cup \overline{D}(0,1/2))
\end{eqnarray*}%
and 
\begin{equation*}
\left\vert p_{1}(z)-q_{1}(z)\right\vert =\left\vert z^{n_{1}}\right\vert
\left\vert q_{1}^{\ast }(z)-p_{1}^{\ast }(z)\right\vert \leq 2^{-2}\text{ \
\ \ }(z\in K_{1}).
\end{equation*}

Next, given $q_{1},...,q_{k-1}$, where $k\geq 2$, we choose $n_{k}>\deg
q_{k-1}$ large enough such that 
\begin{equation*}
\left\vert z^{n_{k}}\left( p_{k}-\tsum_{1}^{k-1}q_{j}\right) (1)\right\vert
\leq 2^{-k-1}w(z)\text{ \ on \ }A\cup \overline{D}(0,k/(k+1)),
\end{equation*}%
define 
\begin{equation*}
p_{k}^{\ast }(z)=\left\{ 
\begin{array}{cc}
z^{-n_{k}}(p_{k}-\tsum_{1}^{k-1}q_{j})(z) & \text{if }\left\vert
z\right\vert \geq 1 \\ 
(p_{k}-\tsum_{1}^{k-1}q_{j})(1) & \text{if }\left\vert z\right\vert <1%
\end{array}%
\right. ,
\end{equation*}%
use Mergelyan's theorem to choose a polynomial $q_{k}^{\ast }$ such that 
\begin{equation*}
\left\vert q_{k}^{\ast }-p_{k}^{\ast }\right\vert
<(2^{k+1}d_{k}^{n_{k}})^{-1}\text{ \ on \ }\overline{A}\cup \overline{D}%
(0,k/(k+1))\cup K_{k},
\end{equation*}%
and define $q_{k}(z)=z^{n_{k}}q_{k}^{\ast }(z)$. Thus 
\begin{eqnarray*}
\left\vert q_{k}(z)\right\vert &\leq &\left\vert z^{n_{k}}\right\vert
\left\vert q_{k}^{\ast }(z)-p_{k}^{\ast }(z)\right\vert +\left\vert
z^{n_{k}}p_{k}^{\ast }(z)\right\vert \\
&\leq &2^{-k-1}+\left\vert z^{n_{k}}\left( p_{k}-\sum_{1}^{k-1}q_{j}\right)
(1)\right\vert \\
&\leq &2^{-k-1}+2^{-k-1}w(z) \\
&\leq &2^{-k}w(z)\text{ \ \ \ \ \ \ \ \ \ }(z\in \overline{A}\cup \overline{D%
}(0,k/(k+1)))
\end{eqnarray*}%
and 
\begin{equation*}
\left\vert p_{k}(z)-\sum_{1}^{k}q_{j}(z)\right\vert =\left\vert
z^{n_{k}}\right\vert \left\vert q_{k}^{\ast }(z)-p_{k}^{\ast }(z)\right\vert
\leq 2^{-k-1}\text{ \ \ }(z\in K_{k}).
\end{equation*}%
It is now easy to see that the series $\sum q_{n}$ converges locally
uniformly on $\mathbb{D}$ to a holomorphic function $f$ such that $%
\left\vert f\right\vert \leq w$ on $A$, and that 
\begin{equation*}
\left\vert p_{k}-S_{n_{k+1}-1}\right\vert =\left\vert
p_{k}-\sum_{1}^{k}q_{j}\right\vert \leq 2^{-k-1}\text{ \ on \ }K_{k}\text{ \
\ \ }(k\in \mathbb{N}).
\end{equation*}%
Thus $f\in \mathcal{U}(\mathbb{D},0)$, as claimed.
\end{proof}

\bigskip

A related result for the strip $S$, given below, will be used in the proof
of Theorem \ref{con}.

\begin{proposition}
\label{ex2}Let $A$ be a bounded subset of $S$ such that $\overline{A}\cap
\partial S=\{\pm 1\}$, and let $w:S\rightarrow (1,\infty )$ be a continuous
function such that $w(z)\rightarrow \infty $ as $z\rightarrow \pm 1$. Then
there exists $f\in \mathcal{U}(S,0)$ such that $\left\vert f\right\vert \leq
w$ on $A$. In particular, this is true for $w=e^{h}$, where $h$ is any
positive harmonic function on $S$ that tends to $\infty $ at $\pm 1$.
\end{proposition}

\begin{proof}
Let $F_{+}:\{\func{Re}z<1\}\rightarrow \mathbb{D}$ be a conformal map such
that $F_{+}(0)=0$ and with boundary limit $F_{+}(1)=1$, and let $%
F_{-}(z)=F_{+}(-z)$. Thus $F_{-}$ is a conformal map from $\{\func{Re}z>-1\}$
to $\mathbb{D}$ and $F_{-}(-1)=1$. We exhaust $S$ by the rectangles%
\begin{equation*}
R_{n}=\{\left\vert \func{Re}z\right\vert \leq n/(n+1),\text{ }\left\vert 
\func{Im}z\right\vert \leq n\}\text{ \ \ \ \ \ \ \ }(n\in \mathbb{N)}.
\end{equation*}%
We may assume that $A$ is closed relative to $S$ and that $\overline{A}\cup
R_{n}$ has connected complement for each $n$. Let $w_{n}=\max_{R_{n}}w$.
Also, let $K_{n}$ and $p_{n}$ be as in the proof of Proposition \ref{ex},
except that the sets $K_{n}$ now lie outside $S$ rather than $\mathbb{D}$.

We inductively define a sequence of polynomials $(q_{n})$ as follows. Given $%
k\in \mathbb{N}$ and $q_{1},...,q_{k-1}$, let $m_{k-1}$ denote the degree of 
$\sum_{1}^{k-1}q_{j}$. (We define $m_{0}=0$.) By Cauchy's estimates we can
choose $\delta _{k}\in (0,1)$ small enough such that, if $g$ is holomorphic
on $\mathbb{D}$ and $\left\vert g\right\vert <\delta _{k}$ on $\overline{D}%
(0,1/2)$, then 
\begin{equation}
\left\vert S_{N}(g,0)\right\vert \leq 2^{-k}\text{ \ \ \ on \ }K_{1}\cup
...\cup K_{k}\text{ \ \ \ }(N=0,1,...,m_{k-1}).  \label{O}
\end{equation}%
Since $\left\vert F_{\pm }\right\vert <1$ and $w(z)\rightarrow \infty $ as $%
z\rightarrow \pm 1$, we can choose $n_{k}\in \mathbb{N}$ large enough so
that 
\begin{equation}
\left\vert F_{\pm }(z)\right\vert ^{n_{k}}\left\vert \left(
p_{k}-\sum_{1}^{k-1}q_{j}\right) (\pm 1)\right\vert \leq 2^{-k-2}\delta _{k}%
\frac{w(z)}{w_{k}}\text{ \ \ \ \ }(z\in A\cup R_{k})  \label{st}
\end{equation}%
and also that $b_{k},c_{k}\in D(0,2^{-k-1}\delta _{k})$, where%
\begin{equation*}
b_{k}=\left\{ F_{-}(1)\right\} ^{n_{k}}\left(
p_{k}-\sum_{1}^{k-1}q_{j}\right) (-1),\text{ \ \ }c_{k}=\left\{
F_{+}(-1)\right\} ^{n_{k}}\left( p_{k}-\sum_{1}^{k-1}q_{j}\right) (1).
\end{equation*}%
The function $p_{k}^{\ast }$ defined by 
\begin{equation*}
p_{k}^{\ast }(z)=\left\{ 
\begin{array}{cc}
p_{k}(z)-\sum_{1}^{k-1}q_{j}(z)+b_{k} & (\func{Re}z\geq 1) \\ 
\begin{array}{c}
\left\{ F_{+}(z)\right\} ^{n_{k}}\left( p_{k}-\sum_{1}^{k-1}q_{j}\right) (1)
\\ 
\;\;\;\;\;\;\;+\left\{ F_{-}(z)\right\} ^{n_{k}}\left(
p_{k}-\sum_{1}^{k-1}q_{j}\right) (-1)%
\end{array}
& (\left\vert \func{Re}z\right\vert <1) \\ 
p_{k}(z)-\sum_{1}^{k-1}q_{j}(z)+c_{k} & (\func{Re}z\leq -1)%
\end{array}%
\right.
\end{equation*}%
is continuous at $\pm 1$ and holomorphic outside $\{\left\vert \func{Re}%
z\right\vert =1\}$, so by Mergelyan's theorem we can choose a polynomial $%
q_{k}$ such that 
\begin{equation*}
\left\vert q_{k}-p_{k}^{\ast }\right\vert <2^{-k-1}\delta _{k}\,\ \ \ \ \ 
\text{\ on \ }\overline{A}\cup R_{k}\cup K_{k}.
\end{equation*}%
In view of (\ref{st}),%
\begin{equation*}
\left\vert q_{k}\right\vert \leq \left\vert q_{k}-p_{k}^{\ast }\right\vert
+\left\vert p_{k}^{\ast }\right\vert \leq 2^{-k-1}\delta _{k}+2^{-k-1}\delta
_{k}\frac{w(z)}{w_{k}}\text{ \ \ \ \ \ \ \ on }A\cup R_{k},
\end{equation*}%
so%
\begin{equation*}
\left\vert q_{k}\right\vert \leq 2^{-k}w\text{ \ on \ }A,\text{ \ \ \ \ \ \
\ }\left\vert q_{k}\right\vert \leq 2^{-k}\delta _{k}\text{ \ on \ }R_{k}.
\end{equation*}%
Also,%
\begin{eqnarray}
\left\vert p_{k}-\sum_{1}^{k}q_{j}\right\vert &\leq &\left\vert
q_{k}-p_{k}^{\ast }\right\vert +\left\vert p_{k}^{\ast }-\left(
p_{k}-\sum_{1}^{k-1}q_{j}\right) \right\vert  \notag \\
&\leq &2^{-k-1}\delta _{k}+\max \{\left\vert b_{k}\right\vert ,\left\vert
c_{k}\right\vert \}\leq 2^{-k}\text{ \ \ \ on \ }K_{k}.  \label{xx}
\end{eqnarray}%
We can clearly also arrange that the sequence $(\delta _{k})$ is decreasing.

It follows that the series $\sum q_{k}$ converges locally uniformly on $S$
to a holomorphic function $f$ satisfying $\left\vert f\right\vert \leq w$ on 
$A$. Further, $\left\vert \sum_{k}^{\infty }q_{j}\right\vert <\delta _{k}$
on $R_{k}$, which contains $\overline{D}(0,1/2)$, so%
\begin{equation*}
\left\vert p_{k}-S_{m_{k-1}}(f,0)\right\vert \leq \left\vert
p_{k}-\sum_{1}^{k-1}q_{j}\right\vert +\left\vert
S_{m_{k-1}}(\sum_{k}^{\infty }q_{j},0)\right\vert \leq 2^{1-k}\text{ \ \ on
\ }K_{k},
\end{equation*}%
by (\ref{xx}) and (\ref{O}). Thus $f\in \mathcal{U}(S,0)$, as required.
\end{proof}

\bigskip

\begin{proof}[Proof of Theorem \protect\ref{con}]
The notion of minimal thinness at a boundary point of $S$ is defined in the
same way as for $\mathbb{D}$, except that we now use the Poisson kernel for $%
S$. Let $h$ be a positive harmonic function on $S$ such that $%
h(z)\rightarrow \infty $ as $z\rightarrow \pm 1$, let $w=e^{h}$, and let $A$
be a bounded, relatively closed subset of $S$ such that $\overline{A}\cap
\partial S=\{\pm 1\}$ and $S\backslash A$ is minimally thin at $\pm 1$.
Next, let $f\in \mathcal{U}(S,0)$ be as in Proposition \ref{ex2}.

Part (i) of Theorem \ref{con} now follows from Corollary \ref{cor}, and the
conformal invariance of harmonicity and minimal thinness.

To prove part (ii) we choose a conformal mapping $F:S\rightarrow S$ which
sends two (distinct) points of $\{\func{Re}z=1\}$ to the boundary points $%
\pm 1$. The argument used to prove Corollary \ref{cor} is readily adapted to
show that, if $f_{1}\in \mathcal{U}(S,0)$ and $h$ is a positive harmonic
function on $S$, then the set $\{\left\vert f_{1}\right\vert \geq e^{h}\}$
cannot be minimally thin at more than one point of $\{\func{Re}z=1\}$. Part
(ii) now follows again from the conformal invariance of harmonicity and
minimal thinness.
\end{proof}

\bigskip\

\bigskip

\noindent School of Mathematical Sciences,

\noindent University College Dublin,

\noindent Belfield, Dublin 4, Ireland.

\noindent e-mail: stephen.gardiner@ucd.ie

\end{document}